\documentclass[12pt]{amsart}
\usepackage{amsmath,amscd,amssymb,amsfonts}

\usepackage[hypertex, pagebackref=true]{hyperref}
\hypersetup{colorlinks=false}

 \usepackage[active]{srcltx}

       \usepackage{color}

\def\lct {{\rm lct}}
\def\bL{\mathbb L}
\newcommand{\bA}{{\mathbb A}}
\newcommand{\bC}{{\mathbb C}}

\newcommand{\bQ}{{\mathbb Q}}
\newcommand{\bZ}{{\mathbb Z}}

\newcommand{\cB}{{\mathcal B}}

\newcommand{\ra}{\rightarrow}

\def\lam{\lambda}

\theoremstyle{plain}
\newtheorem{thm}{Theorem}[section]
\newtheorem{cor}[thm]{Corollary}

\newtheorem{lem}[thm]{Lemma}
\newtheorem{prop}[thm]{Proposition}

\theoremstyle{definition}
\newtheorem{df}[thm]{Definition}
\newtheorem{rem}[thm]{Remark}
\newtheorem{example}[thm]{Example}
\newtheorem{subs}[section]{}

\title[Log canonical thresholds of quasi-ordinary hypersurfaces]{Log canonical thresholds of quasi-ordinary hypersurface singularities}
\author{Nero Budur}
\address{Department of Mathematics,
University of Notre Dame, 255 Hurley Hall, IN 46556, USA} \email{nbudur@nd.edu}
\author{Pedro D. Gonz\'alez-P\'erez}
\address{ICMAT.  Fac. de CC. Matem\'aticas.
Univ. Complutense de Madrid.
Plaza de las Ciencias 3. 28040. Madrid. Spain.}\email{pgonzalez@mat.ucm.es}
\author{Manuel Gonz\'{a}lez Villa}
\address{Fac. de CC. Matem\'aticas.
Univ. Complutense de Madrid.
Plaza de las Ciencias 3. 28040. Madrid. Spain
 and Match, Univ. Heidelberg, Im Neuenheimer Feld 288, 69120 Heidelberg, Germany
}\email{mgv@mat.ucm.es, villa@mathi.uni-heidelberg.de}

\thanks {The first author is supported by the NSA grant H98230-11-1-0169. The second and third authors are supported by MCI-Spain grant MTM2010-21740-C02.}

\begin{document}

\begin{abstract} The log canonical thresholds of irreducible quasi-ordinary hypersurface singularities are computed, using an explicit list of pole candidates for the motivic zeta function found by the last two authors.
\end{abstract}

\maketitle



\begin{subs}\label{subsLCTintro}
Let $f\in\bC[x_1,\ldots ,x_{d+1}]$ be a non-zero polynomial vanishing at the origin in $\bC^{d+1}$. Denote by $Z$ the zero locus of $f$ in a small open neighborhood $U$ of the origin. Consider a log resolution $\mu: Y\ra U$
of $Z$ that is an isomorphism above the complement of $Z$, and let $E_i$ for $i\in J$ be the
irreducible components of $\mu^{-1}(Z)$. Denote by $a_i$
the order of vanishing of $f\circ\mu$ along $E_i$, and by $k_i$
the order of vanishing of the determinant of the Jacobian of $\mu$
along $E_i$. 

The {\it log canonical threshold of $f$ at the origin} is defined as
$$
\lct _0(f):=\min \left\{\frac{k_i+1}{a_i}\ |\ i\in J \right\}.
$$
This is independent of the choice of log resolution. A polynomial $f$ is {\it log canonical at $0$} if $\lct _0(f)=1$. The definition of  the log canonical threshold extends similarly to the case of a germ of complex analytic function $f: (\bC^{d+1}, 0) \to  (\bC, 0)$.

The log canonical threshold is an interesting local invariant of the singularities of $Z$ (the smaller the log canonical threshold is, the worse the singularities of $Z$ are) with connections with many other concepts, see \cite{B-survey, Kollar, Lazarsfeld}. For example, the log canonical threshold of $f$ is the smallest number $c>0$ such that $|f|^{-2c}$ is not locally integrable. It is also the smallest jumping number of $f$,  the negative of the biggest root of the Bernstein-Sato polynomial of $f$, and in certain cases it is a spectral number of $f$. The log canonical threshold can be computed in terms of jet spaces of $\bC^{d+1}$ and $Z$, \cite{Mu}. Furthermore, the set of log canonical thresholds when $d$ is fixed but $f$ varies is known to verify the ascending chain condition, \cite{DEM}.

In this note we give a formula for the log canonical threshold of an irreducible quasi-ordinary polynomial in terms of the associated characteristic exponents, see Theorem \ref{main}.  This result generalizes the well-known case of plane curves singularities, see Example \ref{exCurves}. Unlike the curve case, the log canonical threshold of a quasi-ordinary hypersurface can involve the second characteristic exponent, not only the first one.

\end{subs}

\begin{subs}\label{subsQO}

A germ $(Z,0)$ of an equidimensional  complex analytic variety of dimension $d$
 is \textit{quasi-ordinary} (q.o.) if there exists a
finite projection $\pi: (Z,0) \rightarrow (\bC^d,0)$ that is a
local isomorphism outside a normal crossing divisor. If $(Z,0)$
is a q.o. hypersurface, there is an embedding $(Z,0) \subset (\bC^{d+1},
0) $ defined by an equation $f= 0$,  where $f \in \bC \{ x_1,
\dots, x_d  \} [{y}]$ is a \textit{q.o.~polynomial}, that is,  a
Weierstrass polynomial in ${y}$  with discriminant $\Delta_{{y}} f$ of the
form $\Delta_{{y}} f = x^\delta  u$, for a unit $u$ in the
ring $ \bC \{ x \}$ of convergent power series in the variables $x=
(x_1, \dots, x_d)$ and $\delta \in \bZ^d_{\geq 0}$.     In these coordinates
the  projection $\pi$ is the restriction of the projection
\begin{equation}
 \bC^{d+1}\to \bC^d , \quad (x_1, \dots, x_d,
y) \mapsto (x_1,\dots, x_d).
\end{equation}
The Jung-Abhyankar theorem guarantees that the roots of a q.o.~
polynomial $f$, called {\it q.o.~ branches}, are fractional power
series in the ring $\bC \{ x^{1/m}\}$, for some integer $m \geq 1$, see
\cite{Abhyankar}. Denoting by $K$ the field of fractions of $\bC \{ x_1, \dots, x_d \}$,
if $\tau \in \bC \{ x_1^{1/m}, \dots, x_d^{1/m} \} $ is a q.o.~branch then  the minimal polynomial $F \in K[y]$ of $\tau$
over $K$ has coefficients in    the ring
$\bC \{ x_1, \dots, x_d \}$ and defines the q.o.~hypersurface parametrized by $\tau$.

In this paper we suppose that the germ $(Z,0)$ is analytically irreducible, that is,
the polynomial  $f $ is irreducible in  $\bC \{ x_1, \dots, x_d  \} [{y}]$.
The geometry of an irreducible q.o. polynomial often expresses in term of the 
combinatorics of the corresponding {\it characteristic exponents} which we recall next.

If $\alpha, \beta \in \bQ^d$ we consider the preorder relation given by
$\alpha \leq \beta$ if $\beta \in \alpha + \bQ^d_{\geq 0}$. We set also $\alpha < \beta$
if       $\alpha \leq \beta$ and $ \alpha \ne  \beta$. The notation $\alpha \nleq
\beta$ means that the relation $\alpha \leq \beta$ does not hold. In $\bQ^d
\sqcup \{ \infty \}$ we set that $\alpha < \infty$.

\begin{prop} \cite[Proposition 1.3]{Gau}
\label{expo}   
If $\zeta = \sum c_{\lambda} x^\lambda \in  \bC \{ x^{1/m}\}$ is a q.o.~branch
there exist unique vectors $\lambda_1, \dots,
\lambda_g \in \bQ^d_{\geq 0}$ such that $\lambda_1 \leq \cdots \leq \lambda_g$ and
the three conditions below hold.
 We set   $\lambda_0 = 0$, $\lambda_{g+1} = \infty$,
  and  introduce the lattices
$M_0 := \bZ^d$, $M_j :=
M_{j-1} + \bZ \lambda_j$, for $j=1, \dots, g$.
\begin{enumerate}
\item [(i)]  We have that      $c_{\lambda_j} \ne 0$ for $j =1, \dots, g$.
\item [(ii)] If $c_{\lambda} \ne 0$ then the vector $\lam$ belongs to the lattice $M_j$, where $j$ is the unique integer
such that  $\lambda_j \leq \lambda$ and
    $\lambda_{j+1} \nleq \lambda$.

\item [(iii)] For $j=1, \dots, g$, the vector $\lambda_j $ does not
belong to $ M_{j-1}$.
\end{enumerate}
   If $\zeta \in \bC \{ x^{1/m}\}$ is a fractional power series satisfying the three conditions above then
$\zeta$ is a q.o.~branch.
\end{prop}

\begin{df}
The vectors $\lambda_1 , \dots, \lambda_g $ in Proposition \ref{expo}
are called the {\em characteristic exponents} 
of the q.o.~ branch $\zeta$. 

\end{df}

We introduce also some numerical invariants associated to the characteristic exponents. 
We denote by $n_j$ the index $[M_{j-1} : M_j]$ for $j=1, \dots, g$.
We have that $e_0 := \deg_y f = n_1 \dots n_g$ (see \cite{Lipman}). We define inductively the integers $e_j$ 
by the formula
$e_{j-1}= n_j e_{j}$ for $j=1, \dots, g$.
We set $\ell_0 =0$. If $ 1\leq j \leq g$ we denote by   $\ell_j$ the number of
coordinates of $\lambda_{j}$  which are different from zero.

We denote by $( \lambda_{j,1}, \dots, \lambda_{j,d})$ the coordinates of the
characteristic exponent $\lambda_j$ with respect to the canonical basis
of $\bQ^d$, and by  $\geq_{\mbox{\rm lex}}$ the lexicographic order.
We assume in this note that 
\begin{equation} \label{lex}
(\lambda_{1,1}, \dots, \lambda_{g,1}) \geq_{\mbox{\rm lex}} \cdots
\geq_{\mbox{\rm lex}} (\lambda_{1,d}, \dots, \lambda_{g,d}),
\end{equation}
a condition which holds after a suitable permutation of the variables $x_1, \dots, x_d$. 

  The q.o.~
branch $\zeta$
 is {\em normalized} if the inequalities (\ref{lex}) hold
and if $\lambda_1$ is not of the form $(\lambda_{1,1}, 0, \dots, 0)$ with
$\lambda_{1,1} < 1$.
Lipman proved that if the q.o. branch is not normalized then there exists a normalized q.o.
branch $\zeta'$ parametrizing the same germ $(Z,0)$ (see \cite[Appendix]{Gau}). 
Lipman and Gau studied q.o.~singularities from a topological view-point. 
They proved that the \textit{embedded topological type} of the hypersurface germ $(Z,
 0) \subset (\bC^{d+1}, 0)$ is classified by the  characteristic exponents of a
 normalized q.o.~branch $\zeta$ parametrizing $(Z,0)$,   see \cite{Gau, Lipman}.
\end{subs}

\begin{subs}
We introduce the following numbers in terms of the characteristic exponents:
\[
A_1 := \frac{1 + \lam_{1,1}}{e_0 \lam_{1,1}},
\, 
A_2 : = \frac{ n_1( 1 + \lam_{2,1})  }{ e_1 ( n_ 1( 1 + \lam_{2,1})  - 1)},
\, \mbox{ and  } 
A_3 :=  \frac{1 + \lam_{2,\ell_{1} +1}}{e_1 \lam_{2,\ell_{1} +1}}
\, \mbox{ if } \,  \ell_1 < \ell_2  .
\]

With the above notations our main result is the following:
\begin{thm} \label{main}
Let $f\in\bC\{x_1,\ldots,x_d\}[y]$ be an irreducible quasi-ordinary polynomial. 
We assume that the associated 
characteristic exponents verify (\ref{lex}). 
Then the log canonical threshold of $f$ at the origin is equal to: 
\begin{equation} \label{eq-lct}
\lct_0 (f) = 
\left\{ 
\begin{array}{lcllll}
\min\{ 1, A_1
 \} & \mbox{ if } & \lam_{1,1} \ne \frac{1}{n_1}, & \mbox{ or if} & g = 1,&
\\
\min\{ A_2, A_3
\} & \mbox{ if } &  \lam_{1,1} = \frac{1}{n_1}, & g> 1 & \mbox{ and } & \ell_1 < \ell_2, 
\\
A_2
 & \mbox{ if } & \lam_{1,1} = \frac{1}{n_1}, & g> 1  & \mbox{ and } &  \ell_1 = \ell_2.
\end{array}
\right.
\end{equation}
The number $\lct_0(f)$ is determined by the embedded topological type of the germ defined by $f=0$ at the origin.  
\end{thm}

\begin{cor} \label{cor-lct} With the hypothesis of Theorem \ref{main}, a singular polynomial $f$ is log canonical if and only if $g =1$ and either 
$  \lam_{1,i}  \in \{1 , \frac{1}{2} \}$ or $ \lam_{1,i} = \frac{1}{n_1}$ for $1\leq i \leq \ell_1$.
\end{cor}

\begin{rem}\label{remNorm} Suppose that $\lam_1=(1/n_1,0,...,0)$.
By the \textit{inversion formulae} of \cite{Lipman} the germ $(Z,0)$ is parametrized by a normalized q.o.~branch $\zeta'$ with characteristic exponents $\lambda'_{i} = (n_1( 1+ \lambda_{i+1,1} - 1/n_1), \lambda_{i+1,2}, \dots, \lambda_{i+1,d})$ for $i=1, \dots, g-1$, in particular
$\lambda'_{i,1} > 1 $. If $f'$ is the quasi-ordinary polynomial defined by $\zeta'$ we get that 
$\lct_0(f) = \lct_0(f')$ since both are square-free and define the same germ. 
\end{rem}

\begin{example}\label{exCurves} If $n=2$ and $f\in\bC\{x\}[y]$ defines a singular  irreducible plane germ
 then 
$\lct_0(f) ={\frac{1+ \lam_1}{e_0 \lam_1}}$. This example is
well-known, see \cite{I77}. The log canonical thresholds of plane curve singularities have been considered several times.  For example \cite{Kuwata} gave a explicit formula for this invariant 
in the case of two branches and explained how to compute it for more branches. 
The case of transversal branches is treated with the help of adjoint ideals in \cite{Egorin}.  
The general non-reduced case is done  in \cite{ACNLMlct}. See also \cite{Aprodu-Naie}. 
\end{example}

\end{subs}

\begin{subs}\label{subsOrder}  {\bf Notations.}
We introduce a
sequence of vectors $\alpha_1, \dots, \alpha_g \in \bQ_{\geq 0}^d$ in terms of 
the characteristic exponents $\lambda_1, \dots, \lambda_g$.
We denote by $( \frac{q_1^{(j)}}{p_1^{(j)}}, \dots, \frac{q_d^{(j)}}{p_d^{(j)}})$ the 
coordinates of $\alpha_j$ in terms of the canonical basis of $\bQ^d$, with $\mathrm{gcd} ({q_i^{(j)}}, {p_i^{(j)}}) =1$.
The coordinates of $\alpha_j$ are defined inductively by 
\[
\frac{q_i^{(1)}}{p_i^{(1)}} := \lambda_{1, i},\mbox{ and }  \frac{q_i^{(j)}}{p_i^{(j)}} := 
p_i^{(1)} \cdots p_i^{(j-1)} (\lambda_{j, i} - \lambda_{j-1, i}).
\]
The sequences $\{\lambda_j \}_{j=1}^g$ and $\{ \alpha_j \}_{j=1}^g$ determine each other and by Proposition \ref{expo}
we get that $p_i^{(j)}$ divides $n_j$ for $1 \leq i \leq d$.


\begin{df}The following formulas define pairs of integers $(B_i^{(j)}, b_i^{(j)})$ for $1\leq i \leq d$ and $1 \leq j \leq g$:
\[ 
\begin{array}{cclcccl}
b^{(1)}_i & := &  {p}^{(1)}_i + {q}^{(1)}_i ,  &   \quad & b^{(j)}_i & := &  {p}^{(j)}_i b_i^{(j-1)}  
+ {q}^{(j)}_i,
\\
B^{(1)}_i & := &  e_0  {q}^{(1)}_i,  & \quad &  B^{(j)}_i & := &  p^{(j)}_i B^{(j-1)}_i + e_{j-1} {q}^{(j)}_i.
\end{array}
\]
\end{df}
\begin{rem}
Notice that $B_i^{ (j)} = 0$ if and only if  $\ell_j < i \leq d$ and in that case $ b_i^{(j)} = 1$.
We have also that $A_1 = \frac{b_1^{(1)}}{ B_1^{(1)}}$, $A_2 =  \frac{b_1^{(2)}}{ B_1^{(2)}}$ and 
$A_3 = \frac{b_{\ell_1 +1} ^{(2)}}{B_{\ell_1 +1}^{(2)}}$.
\end{rem}
\end{subs}

\begin{subs}

 In this section we give some properties of the set of the quotients  $\frac{b_i^{(j)}}{ B_i^{(j)}}$.
 
The following formulas are useful in the discussion below.
The first one is consequence of  Proposition \ref{expo}:
\begin{equation} \label{ell}
{0 = \ell_0 <} \ell_1 \leq \cdots \leq \ell_g \leq d.
\end{equation}
If  $\ell_{j-1} < \ell_{j}$  we deduce from the inequalities (\ref{lex}) that  
\begin{equation} \label{ell-lex}
\lambda_{j, \ell_{j-1} +1} \geq  \cdots \geq \lambda_{j, \ell_{j}} \mbox{ and } \lambda_{j, \ell_{j} +1 } = \cdots = \lambda_{j, d} =0.
\end{equation}


\begin{lem}   We have the following inequalities for  
$1 \leq k \leq g$ and $ \ell_{k-1} < i \leq \ell_k$:
\begin{equation}    \label{eq-a}
 \frac{b_{\ell_{k-1} + 1}^{(k)} }{B_{\ell_{k-1} + 1}^{(k)} } \leq  
 \frac{b_{i}^{(k)} }{B_{i}^{(k)} }, 
\end{equation}
\begin{equation} \label{eq-c}
 \frac{1}{e_{k-1}} <  \frac{b_{i}^{(k)} }{B_{i}^{(k)}}; 
\end{equation}
in addition, if  $\frac{q_{i}^{(k)}}{p_{i}^{(k)}} > \frac{1}{n_k} $ then we have
\begin{equation}    \label{eq-e}
\frac{b_{i}^{(k)} }{B_{i}^{(k)} } \leq 
 \frac{1}{e_{k}};
\end{equation} 
if  $k < g$  and $\frac{q_{i}^{(k)}}{p_{i}^{(k)}} =  \frac{1}{n_k} $ then we have
\begin{equation} \label{eq-f}
  \frac{1}{e_{k}} <  \frac{b_{i}^{(k+1)} }{B_{i}^{(k+1)} }  < \frac{1}{e_{k+1}};  
\end{equation}
and if $k < g$  and 
$\frac{q_{\ell_{k-1} + 1}^{(k)} }{p_{\ell_{k-1} + 1}^{(k)} } = \frac{1}{n_k} $  
 then we have
\begin{equation} \label{eq-b}
 \frac{b_{\ell_{k-1} + 1}^{(k+1)} }{B_{\ell_{k-1} + 1}^{(k+1)} } \leq  
 \frac{b_{i}^{(k+1)} }{B_{i}^{(k+1)} }. 
\end{equation}
\end{lem}
\begin{proof} Notice first that if $\ell_{k-1} < i \leq \ell_k$ then $q_i^{(j)} =0$ for $1 \leq j < k$ 
hence we obtain that $b_i^{(k)} = p_i^{(k)} + q_i^{(k)}$ and $B_i^{(k)} = e_{k-1}  q_i^{(k)}$. 
We deduce (\ref{eq-a}) from  (\ref{ell-lex}) and the definitions.
We get (\ref{eq-c}) from the definitions and the inequality  
\begin{equation} \label{eq-h}
\frac{1}{e_{k-1}} <      \frac{1}{e_{k-1}}  \left(  
1 + \frac{1}{\frac{{q_i^{(k)}}}{p_i^{(k)}}} \right) =   \frac{1}{e_{k}}  \left(  
\frac{1}{n_k} + \frac{1}{n_k \frac{q_i^{(k)}}{ p_i^{(k)}}} \right) =  \frac{b_{i}^{(k)} }{B_{i}^{(k)}}.
\end{equation} 
If in addition $\frac{{q_i^{(k)}}}{p_i^{(k)}} > \frac{1}{n_k}$
then we get that       $\frac{{q_i^{(k)}}}{p_i^{(k)}} \geq  \frac{2}{n_k}$. 
Then we deduce  
the inequality (\ref{eq-e}) from {the expression for  $\frac{b_{i}^{(k)} }{B_{i}^{(k)}}$ given at} formula (\ref{eq-h}) by using  that  $n_k \geq 2$.

If  in addition $k < g$ and  $\frac{q_i^{(k)}}{p_i^{(k)}} = \frac{1}{n_k}$ we get from the definitions that
\begin{equation} \label{eq-i}
 \frac{b_{i}^{(k+1)} }{B_{i}^{(k+1)} } =    \frac{1}{e_{k}}   
 \left(  1 + \frac{1}{ n_k + \frac{q_i^{(k+1)}}{p_i^{(k+1)}}    } \right), 
\end{equation} 
This implies that $  \frac{1}{e_{k}}  < \frac{b_{i}^{(k+1)} }{B_{i}^{(k+1)} }  $.   
By formula (\ref{eq-i}) and the inequalities 
$\frac{q_i^{(k+1)}}{p_i^{(k+1)}} \geq \frac{1}{n_{k+1}}$, $n_k, n_{k+1} \geq 2$  
and  
$e_k = n_{k+1}e_{k+1} $  we deduce that 
\[
\frac{b_{i}^{(k+1)} }{B_{i}^{(k+1)} } =    \frac{1}{e_{k+1}}   
 \left(  \frac{1}{n_{k+1}} + \frac{1}{ n_{k+1} ( n_k + \frac{q_i^{(k+1)} }{p_i^{(k+1)}})    } \right) 
{\leq} 
  \frac{1}{e_{k+1}}   
 \left(  \frac{1}{n_{k+1}} + \frac{1}{ n_{k+1} n_k + 1   } \right) 
 <  \frac{1}{e_{k+1}}.
\]
This proves  that the inequality (\ref{eq-f}) holds.

Finally, notice that 
$  \frac{1}{n_k} \leq \frac{q_i^{(k)}}{p_i^{(k)}} 
\leq \frac{{q_{\ell_{k-1} +1} ^{(k)}}}{p_{\ell_{k-1} +1}^{(k)}}$
by  formula (\ref{ell-lex}) and the definitions. 
If $ \frac{{q_{\ell_{k-1} +1} ^{(k)}}}{p_{\ell_{k-1} +1}^{(k)}} = \frac{1}{n_k}$
it follows that  $\frac{q_i^{(k)}}{p_i^{(k)}} = \frac{1}{n_k}$.   We deduce from this 
and  formula (\ref{eq-i})  that  (\ref{eq-b}) holds.
\end{proof}

It is easy to see from the inductive definition of the pairs   $(b_i^{(k)}, B_i^{(k)})$ that  
\begin{equation}
 \label{ord-vert}
   \frac{b_{i}^{(k)} }{B_{i}^{(k)} }  \leq   \frac{b_{i}^{(k+1)} }{B_{i}^{(k+1)} } 
     \quad \Leftrightarrow \quad q_i^{(k+1)} e_k b_i^{(k)} \leq  q_i^{(k+1)} {B_{i}^{(k)} },
\end{equation}
for  $\ell_{k-1} < i \leq \ell_g $ and $1 \leq k < g$.

\begin{lem} \label{one}
  If $\ell_{k-1} < i \leq  \ell_k$ and  $\frac{q_{i}^{(k)} }{p_{i}^{(k)} } >
\frac{1}{n_k}$ then the following inequality holds
\begin{equation}  \label{ord-v}
  \frac{b_{i}^{(k)} }{B_{i}^{(k)} }  \leq  \frac{b_{i}^{(j)} }{B_{i}^{(j)} } 
\quad \mbox{ for } 1 \leq k \leq j \leq g.
\end{equation}
\end{lem}
\begin{proof}
We set $R_j := B_i^{(j)} -  e_{j} b_i^{(j)}$. 
By the equivalence    (\ref{ord-vert})  it is enough to prove that the inequality 
$R_j  \geq 0$ holds for $k \leq j \leq g-1$. 
We prove this by induction. 

For $j = k$  we have the equivalences:
\[
e_k b_i^{(k)} \leq B_i^{(k)} 
{\Leftrightarrow} 
e_k (p_i^{(k)} + q_i^{(k)} ) \leq e_{k-1} q_i^{(k)}
{\Leftrightarrow} 
 {p_i^{(k)}} + {q_i^{(k)}}   \leq n_k q_i^{(k)} 
\Leftrightarrow
\frac{1}{n_k -1} \leq  \frac{q_i^{(k)}}{p_i^{(k)}}.
\]
We deduce that 
the inequality 
\begin{equation} \label{uso-3}
R_k = B_i^{(k)}  - e_k b_i^{(k)} \geq 0  
\end{equation}
holds since $n_k \geq 2$ and  $\frac{q_i^{(k)}}{p_i^{(k)}} \geq \frac{2}{n_k}$ by hypothesis.

Assume that  $k < j$ and  $R_{j-1} \geq 0$. 
Using that $e_{j-1} = n_j e_j$,  we get the following inequalities:  
\begin{equation} \label{uso-4}
  B_i^{(j-1)} - e_j b_i^{(j-1)}  \geq    B_i^{(j-1)} - e_{j-1} b_i^{(j-1)}  =  R_{j-1} \geq 0, 
\end{equation}
and 
\[
R_j = p_i^{(j)} \left( B_i^{(j-1)} - e_j b_i^{(j-1)} \right) + e_j q_i^{(j)} (n_j -1) 
\, 
\stackrel{\mbox{(\ref{uso-4})}}{\geq}
\,
 e_j q_i^{(j)} (n_j -1)
\,
\geq 0.
\]
This completes the proof.
\end{proof}

\begin{lem} \label{two}
     If $\ell_{k-1} < i \leq  \ell_k$ and if $\frac{q_{i}^{(k)} }{p_{i}^{(k)} } =
\frac{1}{n_k}$ then the following inequality holds
\begin{equation}  \label{ord-v-esp}
  \frac{b_{i}^{(k+1)} }{B_{i}^{(k+1)} } \leq  \frac{b_{i}^{(j)} }{B_{i}^{(j)} } 
 \quad \mbox{ for } 
1 \leq k \leq j \leq  g. 
\end{equation}
\end{lem} 
\begin{proof} To compare  $\frac{b_i^{(k+1)}}{B_i^{(k+1)}}$ and $\frac{b_i^{(k)}}{B_i^{(k)}}$ 
we use the expressions (\ref{eq-i}) and (\ref{eq-h}).  

By (\ref{ord-vert}) it is enough  to prove that 
$R_j :=  B_i^{(j)} - e_j b_i^{(j)}  \geq 0$ for $k < j < g$. We prove this by induction on $j$. 
The inequality  $R_{k+1} \geq 0$  is equivalent to 
\begin{equation} \label{eq-j}
 e_{k+1} ( p_i^{(k+1)} b_i^{(k)} + q_i^{(k+1)} )   \leq p_i^{(k+1)} B_i^{(k)} + e_k q_i^{(k+1)}.   
\end{equation}
By hypothesis we have $B_i^{(k)} = e_{k-1}$ and $b_i^{(k)} = 1 + n_k$ hence (\ref{eq-j}) holds
since $e_{k-1} = n_k e_k = n_k n_{k+1} e_{k+1}$ and $n_k, n_{k+1} \geq 2$. 

If $k+1 < j < g$ then we deduce from the induction hypothesis that $R_j \geq 0$ as in Lemma \ref{one}.
\end{proof}

We set
\begin{equation} \label{bset}
 \cB := \{ 1 \} \cup \{ \frac{b_{i}^{(j)} }{B_{i}^{(j)} } \mid 1 \leq i \leq \ell_j \mbox{ and } 
     1 \leq j \leq  g \} \subset \bQ_{\geq 0}. 
\end{equation}

\begin{prop} \label{min} The minimum of the set $\cB$ is the number defined by the right-hand side of formula 
(\ref{eq-lct}).  
\end{prop}
\begin{proof}
    We deal first with the case $\frac{q_1^{(1)}}{p_1^{(1)}} > \frac{1}{n_1}$. 

If    $1 \leq i \leq \ell_1$ and
$\frac{q_i^{(1)}}{p_i^{(1)}} > \frac{1}{n_1}$ we get 
the following inequalities for $1 \leq j \leq g$:
\[
A_1=\frac{b_1^{(1)}}{B_1^{(1)}} \stackrel{ \mbox{(\ref{eq-a})} }{ \leq }
\frac{b_i^{(1)}}{B_i^{(1)}}  \stackrel{ \mbox{(\ref{ord-v})} }{ \leq } \frac{b_i^{(j)}}{B_i^{(j)}}.
\]
 If    $1 <i \leq \ell_1$ and
$\frac{q_i^{(1)}}{p_i^{(1)}} = \frac{1}{n_1}$ we obtain that 
\[
A_1=\frac{b_1^{(1)}}{B_1^{(1)}}   
\, 
\stackrel{ \mbox{(\ref{eq-e})} }{ \leq }
\,
\frac{1}{e_1}
\, 
\stackrel{ \mbox{(\ref{eq-f})} }{ < } 
\,
\frac{b_i^{(2)}}{B_i^{(2)}} 
\, 
 \stackrel{ \mbox{(\ref{ord-v-esp})} }{ \leq } 
\,
\frac{b_i^{(j)}}{B_i^{(j)}},          \mbox{ for }  1 \leq j \leq g.
\]
Suppose now that $1 < k \leq j \leq g$  and $\ell_{k-1} < i \leq \ell_{k}$.  
We have:
\[
A_1=\frac{b_1^{(1)}}{B_1^{(1)}} 
\, 
\stackrel{ \mbox{(\ref{eq-e})} }{ \leq }
\,
\frac{1}{e_1}   
\, 
\stackrel{k>1}{ \leq }
\, 
\frac{1}{e_{k-1}} 
\, 
\stackrel{ \mbox{(\ref{eq-c})} }{ < } 
\, 
\frac{b_{\ell_{k-1} + 1 }^{(k)}}{B_{\ell_{k-1} + 1}^{(k)}} 
\, 
\stackrel{ \mbox{(\ref{eq-a})} }{ \leq } 
\, 
\frac{b_{i}^{(k)}}{B_{i}^{(k)}}
\, 
\stackrel{ \mbox{(\ref{ord-v})} }{ \leq }
\,
 \frac{b_i^{(j)}}{B_i^{(j)}}.
\]
Formula (\ref{ord-v}) in the line above only applies if $\frac{q^{(k)}_{i}}{p^{(k)}_{i}}>\frac{1}{n_k}$. 
Otherwise $\frac{q^{(k)}_{i}}{p^{(k)}_{i}} =\frac{1}{n_k}$ and we use  that 
\[
 \,
\frac{1}{e_1}   
\, 
\stackrel{k>1}{ < }
\, 
\frac{1}{e_{k}} 
\, 
\stackrel{ \mbox{(\ref{eq-c})} }{ < } 
\, 
\frac{b^{(k+1)}_{i}}{B^{(k+1)}_{i}} 
\,
\stackrel{ \mbox{(\ref{ord-v-esp})} }{ \leq } 
\,
 \frac{b_i^{(j)}}{B_i^{(j)}}.
 \]
This finishes the proof in the case  $\frac{q_1^{(1)}}{p_1^{(1)}} > \frac{1}{n_1}$.

We suppose now that  $\frac{q_1^{(1)}}{p_1^{(1)}} = \frac{1}{n_1}$.
By (\ref{ell-lex}) it follows that  $\frac{q_i^{(1)}}{p_i^{(1)}} = \frac{1}{n_1}$ 
for  $1 \leq i \leq \ell_1$. 
We get the inequalities for $1 \leq i \leq \ell_1$ and $1 \leq j \leq g$,
\begin{equation}    \label{eq-k}
A_2=\frac{b_{1}^{(2)}}{B_{1}^{(2)}} 
\, 
\stackrel{ \mbox{(\ref{eq-b})} }{ \leq }
\,
\frac{b_{i}^{(2)}}{B_{i}^{(2)}}  
\, 
\stackrel{ \mbox{(\ref{ord-v-esp})} }{ \leq }
\,
\frac{b_{i}^{(j)}}{B_{i}^{(j)}}.
\end{equation}
If    $\ell_1< \ell_2$ and $\frac{ q_{ \ell_1+1}^{(2)} }{ p_{\ell_1+1}^{(2)}   } > \frac{1}{n_2}$ then 
we deduce the following inequalities for $ 2  \leq j \leq g $ and $\ell_1< i \leq \ell_2$: 
\[ 
 A_3=\frac{b^{(2)}_{\ell_1+1}}{B^{(2)}_{\ell_1+1}} 
 \, 
 \stackrel{\mbox{(\ref{eq-a})}}{\leq} 
 \,
 \frac{b^{(2)}_i}{B^{(2)}_i}   
\, 
\stackrel{ \mbox{(\ref{ord-v})} }{ \leq }
\,
 \frac{b_i^{(j)}}{B_i^{(j)}}.
\]
If   $\ell_1< \ell_2$ and $\frac{ q_{ \ell_1+1}^{(2)} }{ p_{\ell_1+1}^{(2)}   } = \frac{1}{n_2}$
then for $ 2  \leq j \leq g $ and $\ell_1< i \leq \ell_2$ we get      
$\frac{ q_{i}^{(2)} }{ p_{i}^{(2)}   } = \frac{1}{n_2}$  and 
\[
  A_2=\frac{b_{1}^{(2)}}{B_{1}^{(2)}}  
\,
\stackrel{ \mbox{(\ref{eq-f})} }{ < }
\,
\frac{1}{e_2}
\,
\stackrel{ \mbox{(\ref{eq-f})} }{ < }
\,
\frac{b^{(3)}_{\ell_1+1}}{B^{(3)}_{\ell_1+1}} 
\, 
\stackrel{ \mbox{(\ref{eq-b})} }{ \leq }
\,
\,
\frac{b^{(3)}_{i}}{B^{(3)}_{i}} 
\, 
\, 
\stackrel{ \mbox{(\ref{ord-v-esp})} }{ \leq }
\,
\frac{b^{(j)}_i}{B^{(j)}_i}.   
\]
For $k\geq3$ and $\ell_{k-1} < i \leq \ell_k$ we have that
\[
A_2=\frac{b_{1}^{(2)}}{B_{1}^{(2)}}  
\,
\stackrel{ \mbox{(\ref{eq-f})} }{ < }
\,
\frac{1}{e_2}
\,
\stackrel{ k \geq 3 }{ \leq }
\,
\frac{1}{e_{k-1}}
\,
\stackrel{ \mbox{(\ref{eq-c})} }{ < }
\,
\frac{b_{i}^{(k)}}{B_{i}^{(k)}}.
\]
The remaining candidates for the minimum of $\cB$ are discarded by (\ref{eq-a}), (\ref{ord-v}), and (\ref{ord-v-esp}). This completes the proof.
\end{proof}

\end{subs}

\begin{subs}

In this note we use a relation between the log canonical threshold and the poles of the motivic zeta function. 

Let $f$ be as in Section \ref{subsLCTintro}. The {\it local motivic zeta function} and the {\it local topological zeta function of} $f$ of Denef and
Loeser  (see for example \cite{Denef-LoeserBarca}) are
\[
Z_{mot, f}(T)_0:=\sum_{\emptyset \ne I\subseteq J}(\bL-1)^{|I|-1}[E_I^{\circ}\cap \mu^{-1}(0)]\cdot\prod_{i\in I}\frac{\bL^{-(k_i+1)}T^{a_i}}{1 - \bL^{-(k_i+1)}T^{a_i}},
\]
\[
Z_{top, f}(s)_0:=\sum_{\emptyset \ne I\subseteq J}\chi(E_I^{\circ}\cap \mu^{-1}(0))\cdot\prod_{i\in I}\frac{1}{a_is+k_i+1},
\]
where $E_I^{\circ}=(\cap_{i\in I}E_i)-\cup_{i\not\in I}E_i$, the symbol $[\cdot]$ represents  the class of $\cdot$ in the Grothendieck ring $K_0({\rm Var}_\bC)$ of complex algebraic varieties, $\bL$ is the class $[\bA^1]$, and
$\chi$ is the Euler-Poincar\'{e} characteristic.  $Z_{mot, f}(T)_0$ and $Z_{top, f}(s)_0$ are
independent of the choice of the log resolution $\mu$. The set of poles of $Z_{mot,f}(\bL^{-s})_0$ and the set of poles of $Z_{top, f}(s)_0$ are subsets of $\{-(k_i+1)/a_i\ |\ i\in J \}$.

To compute the log canonical threshold of irreducible quasi-ordinary singularities we will use the following result.

\begin{thm}[\cite{Halle-Nicaise} p.18; see also \cite{VZ} 2.7 and 2.8, or \cite{Halle-Nicaise2} 6.3]\label{thmLCTBfct} The biggest pole of $Z_{mot,f}(\bL^{-s})_0$ is equal to $-\lct _0(f)$.
\end{thm}





\end{subs}
\begin{subs} 
We recall some results obtained by the last two authors in \cite{GP-GV}. We use notations of Section \ref{subsQO}
and also the definition of the set $\cB$ in formula (\ref{bset}).
The following result follows from \cite[Corollary 3.17]{GP-GV}. 

\begin{thm}\label{thmPolesZeta} If $f\in \bC\{x_1, \dots, x_d\}[y]$ is an irreducible  quasi-ordinary polynomial then the poles of  $Z_{mot,f}(\bL^{-s})_0$  are contained in the set 
$\{ -\frac{b}{B} \mid \frac{b}{B} \in  \cB \}$. 
\end{thm}

\begin{rem} \label{key} 
 Theorem \ref{thmPolesZeta} is proved by giving a formula for the motivic zeta function in terms of the contact of the jets of arcs with $f$.
The proof uses the change of variable formula for motivic integrals applied to a particular log resolution of $f$. 
This log resolution $ \mu \colon Y \to U \subset \bC^{d+1}$, 
is built as a composition of toric modifications in \cite{GP-Fourier}. If 
${b_i^{(j)}}/{B_i^{(j)}} \in \cB$ then there exists an exceptional divisor $E_i^{(j)}$ of this log resolution 
such that $B_i^{(j)}$ (respectively, $b_i^{(j)} $ minus one)  is the order of 
vanishing of the pull-back of $f$ (respectively, of the determinant of the Jacobian of $\mu$) along $E_i^{(j)}$.
This is a consequence of Corollary 3.17,  Remark 3.19, and Lemma 9.11 of \cite{GP-GV}. 
\end{rem}

\begin{rem}
Notice that the pairs 
$(B_i^{(j)}, b_i^{(j)}) = (0,1)$ do not contribute to a candidate pole  of  $Z_{mot,f}(\bL^{-s})_0$. 
The list of candidate poles indicated in Theorem \ref{thmPolesZeta} arises also in \cite{ACNLM-AMS} with a different method, see \cite{GP-GV} for a comparison.
\end{rem}
\end{subs}
\begin{subs} We prove the main results of this note: 

\noindent 
{\bf Proof of Theorem \ref{main}.} 
Since $\lct _0(f)$ is by definition the
minimum of ${(k_i+1)}/{a_i}$ for $i\in J$, it follows from  Theorem
\ref{thmLCTBfct}, Theorem \ref{thmPolesZeta}, and Remark \ref{key} that 
$\lct _0(f) = \min \cB$.  The result follows then from 
Proposition  \ref{min}.  \hfill $\Box$

\medskip

\noindent 
 {\bf Proof of Corollary \ref{cor-lct}.} 
If $f$ is singular and log canonical then $1\leq
\frac{b^{(1)}_1}{B^{(1)}_1}=  \frac{1}{e_0} ( 1 + \frac{p_1^{(1)}}{q_1^{(1)}} )$. 
Since $ \frac{q_1^{(1)}}{p_1^{(1)}} \geq \frac{1}{n_1}$ we deduce that 
$
e_0 - 1 = n_1 \dots n_g -1 \leq \frac{p_1^{(1)}}{q_1^{(1)}} \leq n_1$.
This implies that $g=1$. If $n_1=2$ there are two possible cases 
$\lam_1= ( 1, \ldots ,1,1/2,\ldots ,1/2,0,\ldots ,0)$ or $\lam_1 = (1/2,\ldots ,1/2,0,\ldots ,0)$.
If $n_1 >  2$ we must have $p_1^{(1)} = n_1$ and $q_1^{(1)} = 1$, since $p_1^{(1)}$ divides $n_1$. 
By (\ref{ell-lex}) we get that  $\lam_1 = ( \frac{1}{n_1}, \ldots, \frac{1}{n_1}, 0, \ldots, 0)$. 
 \hfill $\Box$

\end{subs}

\begin{subs} We end this note with some examples.

\begin{example}Let $\lam_1=(1/3,1/3)$, $\lam_2=(7/6,2/3)$. A polynomial with these characteristic exponents is for example $f=(z^3-xy)^2-x^3y^2z^2$. We have $n_1=3$ and $n_2=2$. By Theorem \ref{main}, the log canonical threshold comes from $\lam_{2,1}$ and equals $A_2=13/22$. Indeed, $\frac{b^{(1)}_1}{B^{(1)}_1}=2/3$,
 $\frac{b^{(2)}_1}{B^{(2)}_1}=13/22$,  $\frac{b^{(2)}_2}{B^{(2)}_2}=5/8$, and the minimum of these is $13/22$.
\end{example}

\begin{example}
Let us consider a q.o.~polynomial with characteristic exponents $\lam_1= (1/2, 1/2, 0)$ and $\lam_2 = (2/3, 2/3, 11/3)$. For instance $f = (y^2 - x_1 x_2)^3 - (y^2 - x_1 x_2) x_1^6 x_2^6 x_3^{11}$. 
We have that $n_1= 2$ and $n_2= 3$ and  $\cB = \{ 1, 1/2 , 10/21, 14/33 \}$. We get $\lct_0(f) = 14/33 = A_3$. 
\end{example}
\end{subs}

{\bf Acknowledgement.} The first author would like to thank Johns Hopkins University for its hospitality during the writing of this article.

\end{document}